\documentclass[11pt]{amsart}
\usepackage{graphics}
\usepackage{caption}
\usepackage{geometry}
\usepackage{tikz}
\usetikzlibrary{matrix}
\usetikzlibrary{arrows}

\usepackage{amsmath}
\usepackage{amssymb}
\usepackage{amsthm}
\usepackage{vmargin}
\usepackage{hyperref}
\usepackage{dsfont}
\usepackage{lmodern}
\usepackage{textcomp}
\usepackage{enumerate}
\usepackage{bbm}

\makeatletter
\newcommand{\labitem}[2]{%
\def\@itemlabel{\text{(#1)}}
\item
\def\@currentlabel{#1}\label{#2}}
\makeatother

\DeclareFontFamily{U}{mathx}{\hyphenchar\font45}
\DeclareFontShape{U}{mathx}{m}{n}{
      <5> <6> <7> <8> <9> <10>
      <10.95> <12> <14.4> <17.28> <20.74> <24.88>
      mathx10
      }{}
\DeclareSymbolFont{mathx}{U}{mathx}{m}{n}
\DeclareFontSubstitution{U}{mathx}{m}{n}
\DeclareMathAccent{\widecheck}{0}{mathx}{"71}

\newcommand{\PP}{\ensuremath{\mathbb{P}}}

\newcommand{\HH}{\ensuremath{\mathcal{H}}}
\newcommand{\NN}{\ensuremath{\mathbb{N}}}

\newcommand{\TT}{\ensuremath{\mathbb{T}}}

\newcommand{\op}[1]{\operatorname{#1}}

\theoremstyle{plain}
\newtheorem{tw}{Theorem}[section]
\newtheorem{stw}[tw]{Proposition}
\newtheorem{lem}[tw]{Lemma}

\newtheorem{wn}[tw]{Corollary}

\theoremstyle{definition}
\newtheorem{defin}[tw]{Definition}

\theoremstyle{remark}

\newtheorem{uw}[tw]{Remark}

\begin{document}
\author{Mateusz Wasilewski}
\address{Institute of Mathematics of the Polish Academy of Sciences, ul. \'{S}niadeckich 8, 00-956 Warszawa, Poland}
\email{wasilewskimat@gmail.com}

\title{Amalgamated direct sums of operator spaces}

\begin{abstract}
We consider amalgamated direct sums (and their dual counterparts -- fibre products) of operator spaces and study their behaviour with respect to different quantisations (minimal and maximal). We show examples of amalgamated direct sums of two $L^{\infty}$-spaces over a common subspace that are not minimal themselves, thus answering a question posed by Vern Paulsen.
\end{abstract}

\maketitle

\section{Introduction}
An operator space $X$ is a closed subspace of $\op{B}(\HH)$, considered with induced norms on the spaces of complex matrices $\op{M}_n \otimes X \subset \op{M}_n(\op{B}(\HH)) \cong \op{B}(\HH^{\oplus n})$. The sequences of norms $(\|\cdot\|_{n})_{n \in \NN}$ that arise in this way on $(\op{M}_n(X))_{n \in \NN}$ have been characterised by Ruan (cf. \cite[Theorem 2.3.5]{ER}). This abstract characterisation enabled a quick development of the theory of operator spaces, which is by now very rich. On the one hand, it serves as powerful tool for solving problems in the theory of operator algebras (cf. \cite{Bl} and \cite{BRS}) and harmonic analysis (cf. \cite{Ru}). On the other hand, one can try to investigate the theory itself; this is the approach that we will follow. More precisely, we will study quantisations, i.e.~methods of introducing an operator space structure on a given Banach space; we will restrict our attention to the well-known minimal and maximal quantisations (cf. \cite[Examples 2.3, 2.4]{BP}). It is quite well understood how direct sums and quotients behave with respect to these functors. In this paper we would like to begin investigation of amalgamated direct sums and fibre products in this context.

We will begin with recalling useful facts from the theories of Banach and operator spaces in Section \ref{prelim}. In Section \ref{amalg} we provide definitions of amalgamated direct sums and fibre products and prove basic facts about them. In the next section we study quotients/subspaces of minimal/maximal operator spaces, both as an interesting subject in itself and as a preparation for the last section. Section \ref{mainproof} contains the main goal of this paper, which is the following theorem.
\begin{tw}\label{main}
There exist amalgamated $\ell_{\infty}$-direct sums (defined precisely in Section \ref{amalg}) of two $L^{\infty}$-spaces over a common subspace that are not completely isomorphic to minimal operator spaces.
\end{tw}

\section{Preliminaries}\label{prelim}
We will need several well-known facts from the theory of Banach and operator spaces. The background from Banach spaces is more modest, so let us start with it. The next theorem uses the notion of an absolutely summing operator; the relevant background on this topic may be found in \cite{DJT}.
\begin{tw}[Grothendieck, \cite{Gr}]
Let $T:L^{1}(\Omega,\mu) \to \HH$ be a bounded operator, where $\HH$ is a Hilbert space. Then $T$ is absolutely summing.
\end{tw}
\begin{uw}
This is actually an equivalent form of Grothendieck's result. For more on this topic, see \cite{Pi1}. 
\end{uw}
\begin{uw}
It is irrelevant for this particular theorem, but we will deal only with localisable measure spaces (cf. \cite[211G]{Fr}) to keep the duality $(L^1)^{\ast} \cong L^{\infty}$ (cf. \cite[243G]{Fr}).
\end{uw}
Now let us recall basic facts about minimal and maximal quantisation\footnote{By quantisation we mean a method of imposing an operator space structure on a given Banach space.} of a given Banach space $X$, right after we define them.
\begin{defin}
An operator space $X$ is \textbf{minimal} if and only if for every operator space $Y$ any contraction $T:Y \to X$ is a complete contraction.

An operator space $X$ is \textbf{maximal} if and only if for every operator space $Y$ any contraction $T:X \to Y$ is a complete contraction.
\end{defin}
\begin{stw}[cf. {\cite[\textsection 3.3]{ER}}]
Minimal and maximal quantisations are dual to each other, more precisely:
\begin{enumerate}[{\normalfont (i)}]
\item $(\op{MIN}(X))^{\ast} \cong \op{MAX}(X^{\ast})$;
\item $(\op{MAX}(X))^{\ast} \cong \op{MIN}(X^{\ast})$,
\end{enumerate}
where `` $\cong$'' denotes a completely isometric isomorphism. In particular, an operator space $X$ is minimal (maximal) if and only if its dual is maximal (minimal).

Minimality (maximality) is preserved by $\ell_{\infty}$-direct sum ($\ell_1$-direct sum).

Minimality passes to subspaces and maximality descends to quotients.
\end{stw}
We will also need one fact about the row Hilbert space $\mathcal{R}$.
\begin{stw}[cf. {\cite[Theorem 24.2]{Pi2}}]
Row Hilbert space $\mathcal{R}$ is an injective operator space, i.e.~for any pair of operator spaces $Y \subset X$ and a completely contractive operator $T:Y \to \mathcal{R}$ there exists a completely contractive extension $\widetilde{T}:X \to \mathcal{R}$:
\begin{figure}[h]
\centering
\begin{tikzpicture}
\matrix (m) [
			matrix of math nodes,
      row sep=3em,
      column sep=3em,
      text height=1.5ex,
      text depth=0.25ex
      ]
     {
     X & \mathcal{R}  \\
     Y &  \\
      };
      \path[right hook->] 	(m-2-1) edge (m-1-1); 
      \path[->]						(m-2-1) edge node[above] {$T$}       (m-1-2);
      \path[dashed,->]  (m-1-1) edge node[above] {$\widetilde{T}$} (m-1-2);      						

\end{tikzpicture}
\end{figure}
\end{stw}
Now we head to the main part of this paper.
\section{Amalgamated direct sums and fibre products}\label{amalg}
Let us first recall that the $\ell_p$-direct sums of Banach spaces (denoted by $\oplus_{p}$) have their counterparts also in the world of operator spaces, cf. \cite[\textsection 2.6 and Remark 2.7.3]{Pi2}.

We are now ready for the main definition.
\begin{defin}
Let $X$ and $Y$ be two operator spaces with a common subspace $Z$ (it means that there are complete isometries $Z \hookrightarrow X$ and $Z \hookrightarrow Y$). Then we may define the \textbf{amalgamated $\ell_{p}$-direct sum} of $X$ and $Y$ over $Z$ as
$$
(X \oplus_{Z} Y)_{p} := X \oplus_{p} Y \slash \overline{\op{span}}\{(z,-z): z \in Z\}.
$$
Dually, let $X$ and $Y$ be operator spaces equipped with complete quotient maps $\pi_1:X \to Z$ and $\pi_2:Y \to Z$. Then we define the \textbf{fibre $\ell_{p}$-product} of $X$ and $Y$ over $Z$ as
$$
(X \times_{Z} Y)_{p} := \{(x,y) \in X\oplus_{p} Y: \pi_1(x)=\pi_2(y)\}.
$$
\end{defin}
Let us start with proving basic properties of these constructions.
\begin{stw}
Amalgamated direct sums and fibre products are dual to each other:
\begin{enumerate}[{\normalfont (i)}]
\item\label{bla1} If $Z$ is a common subspace of operator spaces $X$ and $Y$ then $Z^{\ast}$ is naturally a quotient of both $X^{\ast}$ and $Y^{\ast}$. Moreover $\left((X \oplus_{Z} Y)_{p}\right)^{\ast} \simeq (X^{\ast} \times_{Z^{\ast}} Y^{\ast})_{p'}$, where $p'=\frac{p}{p-1}$.
\item\label{bla2} If $Z$ is a quotient of both $X$ and $Y$ then $Z^{\ast}$ is a common subspace of $X^{\ast}$ and $Y^{\ast}$. Moreover $\left( (X \times_{Z} Y)_{p}\right)^{\ast} \simeq (X^{\ast} \oplus_{Z^{\ast}} Y^{\ast})_{p'}$, where $p'=\frac{p}{p-1}$.
\end{enumerate}
\end{stw}
\begin{proof}
\eqref{bla1} The first part follows from the fact that if $T:X \to Y$ is a complete isometry then $T^{\ast}:Y^{\ast} \to X^{\ast}$ is a complete quotient map (cf. \cite[\textsection 2.4]{Pi2}). For the second part, we use the duality between quotients and subspaces, namely 
$$
\left(X \oplus_{p} Y \slash \overline{\op{span}}\{(z,-z):z \in Z\}\right)^{\ast} \simeq (\overline{\op{span}}\{(z,-z):z \in Z\})^{\perp} \subset X^{\ast} \oplus_{p'} Y^{\ast}. 
$$
The proof boils down to computation of $(\overline{\op{span}}\{(z,-z):z \in Z\})^{\perp}$. If $(\varphi, \psi) \in X^{\ast} \oplus_{p'} Y^{\ast}$ annihilates all elements of the form $(z,-z)$ then $\varphi_{|Z} = \psi_{|Z}$, which means exactly that $(\varphi, \psi)$ belongs to $(X^{\ast} \times_{Z^{\ast}} Y^{\ast})_{p'}$.

\eqref{bla2} The proof goes exactly along the lines of the above proof of \eqref{bla1}.
\end{proof}
The main aim of this paper is to investigate the behaviour of amalgamated direct sums and fibre products under minimal and maximal quantisations. Let us first deal with easy cases.
\begin{stw}
Let $X$, $Y$, $Z$ be operator spaces.
\begin{enumerate}[{\normalfont (i)}]
\item\label{stw1}
If $Z$ is a quotient of both $X$ and $Y$, with $X$ and $Y$ minimal, then the fibre $\ell_{\infty}$-product $(X \times_Z Y)_{\infty}$ is minimal as well.
\item\label{stw2} Suppose that $X$ and $Y$ are maximal operator spaces with a common subspace $Z$. Then their amalgamated $\ell_{1}$-direct sum is maximal.
\end{enumerate}
\end{stw}
\begin{proof}
\eqref{stw1} If $Z$ is a quotient of both $X$ and $Y$, with $X$ and $Y$ minimal, then the fibre $\ell_{\infty}$-product $(X \times_Z Y)_{\infty}$ is, by definition, a subspace of a minimal space $X \oplus_{\infty} Y$, so it is also minimal.

\eqref{stw2} If $X$ and $Y$ are maximal then, by the preceding proposition, the dual of $(X \oplus_{Z} Y)_{1}$ is $(X^{\ast} \times_{Z^{\ast}} Y^{\ast})_{\infty}$, which is minimal by the first part of the proof.
\end{proof}
Later (in Section \ref{mainproof}) we will provide examples of amalgamated $\ell_{\infty}$-direct sums of minimal spaces that are not minimal themselves. The main tool will be the following observation.
\begin{stw}\label{amalgfib}

\begin{enumerate}[{\normalfont (i)}]
\item\label{nonminam}
Suppose that a Banach space $X$ has the following property: whenever $X \subset L^{\infty}$ then the quotient $\op{MIN}(L^{\infty}) \slash X$ is not completely isomorphic to a minimal operator space.

Then, if $X \subset L^{\infty}(\Omega, \mu)$ and $X \subset L^{\infty}(\Omega', \mu')$ then the amalgamated $\ell_{\infty}$-direct sum $\left(\op{MIN}(L^{\infty}(\Omega, \mu)) \oplus_{X} \op{MIN}(L^{\infty}(\Omega', \mu'))\right)_{\infty}$ is not completely isomorphic to a minimal operator space. 
\item\label{nonmaxfi}
Suppose that a Banach space $X$ has the following property: whenever there is a quotient map $T: L^{1} \to X$ then $\op{ker}T \subset \op{MAX}(L^{1})$ is not completely isomorphic to a maximal operator space.

Then, if there are quotient maps $S: L^{1}(\Omega, \mu) \to X$ and $T:L^{1}(\Omega', \mu') \to X$ then the fibre $\ell_1$-product $\left(\op{MAX}(L^{1}(\Omega,\mu)) \times_{X} \op{MAX}(L^{1}(\Omega', \mu'))\right)_{1}$ is not completely isomorphic to a maximal operator space.
\end{enumerate}
\end{stw}
\begin{proof}
\eqref{nonminam}
The subspace $\{(x,-x): x \in X\} \subset L^{\infty}(\Omega, \mu) \oplus_{\infty} L^{\infty}(\Omega', \mu')$ is isometric to $X$ and $L^{\infty}(\Omega, \mu) \oplus_{\infty} L^{\infty}(\Omega', \mu') \cong L^{\infty}(\Omega \sqcup \Omega', \mu \sqcup \mu')$, so the amalgamated $\ell_{\infty}$-direct sum $\left(\op{MIN}(L^{\infty}(\Omega, \mu)) \oplus_{X} \op{MIN}(L^{\infty}(\Omega', \mu'))\right)_{\infty}$ is completely isometric to a quotient of an $L^{\infty}$-space by an isometric copy of $X$.

\eqref{nonmaxfi} 
By definition, the fibre $\ell_1$-product $\left(\op{MAX}(L^{1}(\Omega,\mu)) \times_{X} \op{MAX}(L^{1}(\Omega', \mu'))\right)_{1}$ is equal to the kernel of the map $U: L^{1}(\Omega, \mu) \oplus_1 L^1(\Omega', \mu') \to X$, given by $U(f,g) = S(f) - T(g)$. Since $S$ is a quotient map, for every $x \in X$ there exists an $f \in L^{1}(\Omega,\mu)$ such that $\|f\| \leqslant \|x\|(1+\varepsilon)$ and $U(f,0) = x$, therefore $U$ is a quotient map. Since $L^{1}(\Omega, \mu) \oplus_1 L^1(\Omega', \mu') \cong L^{1}(\Omega \sqcup \Omega', \mu \sqcup \mu')$, $\op{ker}U$ is not completely isomorphic to a maximal operator space; the result follows.
\end{proof}
Since, by definition, amalgamated $\ell_{\infty}$-direct sums of $L^{\infty}$-spaces are special quotients of minimal operator spaces, we will start with exhibiting a new example of a non-minimal quotient, which will come in handy during the proof of the main theorem.
\section{Quotients/subspaces of minimal/maximal operator spaces}
To make our work meaningful, we should first ensure that there exist quotients of minimal spaces that are not minimal. This is, by duality, equivalent to producing an example of a subspace of a maximal operator space that is not maximal, and this is what we want to recall.
\begin{tw}[Lust-Piquard, Pisier, \cite{LPP}]
Let $(r_n)_{n \in \NN}$ be the Rademacher sequence. The subspace $\overline{\op{span}}(r_n)_{n \in \NN} \subset \op{MAX}(L^{1})$ is completely isomorphic to the Hilbertian operator space $\mathcal{R} + \mathcal{C}$ (for the definition, see \cite[\textsection 9.8]{Pi2}), which is not completely isomorphic to a maximal operator space.
\end{tw}
We will now produce numerous examples of subspaces of $\op{MAX}(L^1)$ that are not maximal, using a theorem of Kalton and Pe\l czy\'{n}ski from Banach space theory. Let us first formulate the result.
\begin{tw}\label{nonmax}
Let $X$ be a subspace of $\op{MAX}(L^1)$ that is not a GT-space (i.e.~there exists an operator $S:X \to \ell_2$ that is not absolutely summing). Then $X$ is not maximal.
\end{tw}
\begin{proof}
Suppose that $X$ is completely isomorphic to a maximal operator space and let $S: X \to \ell_2$ be an operator that is not absolutely summing. By maximality, $S: X \to \mathcal{R}$ is completely bounded. Since we know that $\mathcal{R}$ is an injective operator space, there exists a completely bounded extension $\widetilde{S}:\op{MAX}(L^1) \to \ell_2$, which is absolutely summing by Grothendieck's theorem. It follows that its restriction to $X$, which is equal to $S$, is also absolutely summing, hence we arrive at a contradiction.
\end{proof}
This simple theorem will serve as a basis for obtaining non-maximal subspaces of $\op{MAX}(L^1)$. Let us start with reproving known results, using this technique.
\begin{stw}
The Hardy space $H^{1}(\TT) \subset \op{MAX}(L^1(\TT))$ is not a maximal operator space. Dually, the quotient $\op{MIN}(L^{\infty}(\TT)) \slash H^{\infty}(\TT)$ is not minimal (identified in \cite[\textsection 9.1]{Pi2} with the space of Hankel operators). Any Hilbertian (infinite dimensional) subspace of $\op{MAX}(L^1)$ is not maximal.
\end{stw}
\begin{proof}
According to Theorem \ref{nonmax}, to show non-maximality of $X \subset \op{MAX}(L^1)$, we need to produce an operator $T:X \to \ell_2$ that is not absolutely summing.

If $X$ is Hilbertian then we just take the identity (or an orthogonal projection if $X$ happens to be non-separable)\footnote{Note that Oikhberg \cite{Oi} proved that any homogeneous (an operator space $X$ is \textbf{homogeneous} if every contraction $T:X \to X$ is a complete contraction) Hilbertian subspace of $\op{MAX}(L^1)$ is completely isomorphic to $\mathcal{R} + \mathcal{C}$, and maximal operator spaces are homogeneous, so our result is definitely not new, but the proof is very quick.}.

Let us turn our attention to the Hardy space. We need to exhibit a bounded operator $T: H^{1}(\TT) \to \ell_2$ that is not absolutely summing. Consider the famous ``Paley projection'' $Pf = (\widehat{f}(2^{n}))_{n \in \NN}$; Paley proved (cf. \cite[Theorem 6.7]{Du}) that it is bounded. Yet, it is not absolutely summing because the sequence $(e^{i2^{n} t})_{n \in \NN}$ converges weakly to $0$ in $H^{1}(\TT)$ (Riemann-Lebesgue lemma) and the sequence of its images is an orthonormal basis of $\ell_2$, so it is not norm convergent, and absolutely summing operators are known to be completely continuous (cf. \cite[Theorem 2.17]{DJT}).
\end{proof}
We will now need the aforementioned theorem of Kalton and Pe\l czy\'{n}ski to produce new examples.
\begin{tw}[Kalton, Pe\l czy\'{n}ski, \cite{KP}]\label{KP}
Suppose that $X$ is an infinite-dimensional Banach space that either (cf. \cite[Chapter 11, Chapter 13]{DJT} for both these notions):
\begin{enumerate}[{\normalfont(i)}]
\item has infinite cotype (contains copies of $\ell_{\infty}^{n}$'s uniformly);
\item is K-convex (does not contain copies of $\ell_1^n$'s uniformly).
\end{enumerate}
Assume also that there is a surjective linear operator $T:L^{1} \to X$. Then the kernel $\op{ker}T$ is not a GT-space.
\end{tw}
We will now use this theorem.
\begin{lem}\label{steinhaus}
Suppose that $(s_n)_{n \in \NN}$ is a sequence of i.i.d.~Steinhaus random variables, i.e.~uniformly distributed on the unit circle. Consider the map $T: L^1\to \ell_{\infty}$ given by $Tf := (\int f s_n)_{n \in \NN}$. Then the image of $T$ is equal to $c_0$ and $T: L^1 \to c_0$ is a quotient map.
\end{lem}
\begin{proof}
First of all, the estimate $|\int f s_n| \leqslant \|s_n\|_{\infty} \|f\|_1 = \|f\|_1$ shows that $T$ is a contraction. By independence, $s_n$ are orthonormal, so $T(L^{2}) \subset \ell_{2} \subset c_{0}$. Since $L^2$ is dense in $L^1$, $T(L^1)$ is contained densely in $c_{0}$. The dual operator $T^{\ast}: \ell_1 \to L^{\infty}$ is easily seen to be isometric, so $T$ has closed image and is a quotient map.
\end{proof}
\begin{uw}
One can easily see that, for the above conclusions to hold, it suffices to assume that the random variables $(s_n)$ have mean $0$ (so that the image of $T$ is contained in $c_0$), their law is supported in the unit disc and the unit circle is contained in the support (to ensure that $T^{\ast}$ is isometric).
\end{uw}
\begin{wn}\label{minquot}
The kernel of the map $T: L^{1} \to c_0$ from Lemma \ref{steinhaus} is a non-maximal subspace of $\op{MAX}(L^1)$. Its dual, equal to $\op{MIN}(L^{\infty}) \slash \op{Ran}(T^{\ast})$, is not minimal. Moreover, $\op{Ran}(T^{\ast})$ is equal to $\overline{\op{span}}(s_n)_{n \in \NN}$, so we obtain a very explicit quotient of a minimal operator space that is not minimal. More generally, quotient of $\op{MIN}(L^{\infty})$ by any isomorphic copy of $\ell_1$ is not minimal.
\end{wn}
\begin{proof}
We just need to justify the fact that $\op{Ran}(T^{\ast})$ is weak$^{\ast}$-closed and the last assertion. The former follows from the Krein-\v{S}mulian theorem, because $T^{\ast}$ is a weak$^{\ast}$-continuous isometry. For the latter, if $\iota:\ell_1 \subset L^{\infty}$ is an isomorphic embedding then it dualises to a surjection $\iota^{\ast}: (L^{\infty})^{\ast} \to \ell_{\infty}$. Since $(L^{\infty})^{\ast\ast}$ is a commutative von Neumann algebra, $(L^{\infty})^{\ast}$ is isometric to an $L^{1}$-space\footnote{Every commutative von Neumann algebra is isomorphic to $L^{\infty}(\Omega, \mu)$, where the measure space $(\Omega, \mu)$ is localisable, essentially because the lattice of projections is complete.}. This means that, by Theorem \ref{nonmax} and Theorem \ref{KP}, $\op{ker}(\iota^{\ast}) \cong (\iota(\ell_1))^{\perp} \cong \left( L^{\infty} \slash \iota(\ell_1)\right)^{\ast}$ is not maximal, so $L^{\infty} \slash \iota(\ell_1)$ is, \emph{a fortiori}, not minimal.
\end{proof}

To end this section, let us formulate a result dual to Theorem \ref{KP}:
\begin{tw}\label{kpquot}
Suppose that $X$ is a Banach space such that $X^{\ast}$ satisfies the assumptions of Theorem \ref{KP}. If $\iota:X \to L^{\infty}$ is an isomorphic embedding then $\op{MIN}(L^{\infty}) \slash \iota(X)$ is not minimal.
\end{tw}
\begin{proof}
The isomorphic embedding $\iota: X \to L^{\infty}$ dualises to the surjection $\iota^{\ast}: (L^{\infty})^{\ast} \to X^{\ast}$, whose kernel $\op{ker} \iota^{\ast} \subset \op{MAX}((L^{\infty})^{\ast})$ is not maximal by combination of Theorem \ref{KP} and Theorem \ref{nonmax}.
\end{proof}
\begin{wn}\label{lpquot}
Let $X \subset L^{\infty}$ be isomorphic to an $L^{p}$-space. If $p=\infty$ then the quotient $L^{\infty} \slash X$ is a minimal operator space. Otherwise, $L^{\infty} \slash X$ is not completely isomorphic to a minimal operator space.
\end{wn}
\begin{proof}
If $X$ is isomorphic to an $L^{\infty}$-space then it is $\lambda$-injective\footnote{It means that every contraction from $Y \subset Z$ to $X$ admits an extension to $Z$ of norm not greater than $\lambda$.} as a Banach space (for some $\lambda >0$), so there is a bounded projection $P: L^{\infty} \to X$ (extend $\op{Id}_{X}$ to $L^{\infty}$). This means that the quotient $L^{\infty} \slash X$ may be identified with $\op{Ran}(\op{Id}_X - P)$, so it is a minimal operator space. To wit, the isomorphism $T: \op{Ran}(\op{Id}_{X} - P) \to L^{\infty} \slash X$ is completely bounded (as a restriction of the canonical quotient map) and its inverse is also completely bounded, since $\op{Ran}(\op{Id}_X - P)$ is minimal.

If $X$ is isomorphic to an $L^{p}$-space for $p<\infty$ then $X^{\ast}$ is K-convex for $1<p<\infty$ and has infinite cotype for $p=1$, so Theorem \ref{kpquot} implies that $L^{\infty} \slash X$ is not completely isomorphic to a minimal operator space.  
\end{proof}
\section{Proof of the main theorem}\label{mainproof}
This section is devoted to the proof of Theorem \ref{main} announced in the Introduction. We will rely heavily on the results obtained in the previous section. Let us first make Theorem \ref{main} more explicit:
\begin{tw}\label{nonminamalg}
Let $(\Omega, \PP)$ be a probability space and let $X=Y= L^{\infty}(\Omega, \PP)$. Let $(s_n)_{n \in \NN}$ be, once again, a sequence of i.i.d.~Steinhaus random variables and let $Z= \overline{\op{span}}(s_n)_{n \in \NN}$. Then the amalgamated $\ell_{\infty}$-direct sum $(X\oplus_{Z} Y)_{\infty}$ is not completely isomorphic to a minimal operator space.
\end{tw}
The idea of the proof is very simple: we already know that $\op{MIN}(L^{\infty}(\Omega,\PP)) \slash \overline{\op{span}}(s_n)_{n \in \NN}$ is not minimal (Corollary \ref{minquot}), regardless of the choice of the probability space $(\Omega, \PP)$. We will show that $(X \oplus_Z Y)_{\infty}$ is completely isometric to this kind of space, for some other probability space $(\Omega', \PP')$:
\begin{stw}
Let $X$, $Y$, and $Z$ be as in the statement of Theorem \ref{nonminamalg}. Then there exists a probability space $(\Omega', \PP')$ and a sequence $(s_n')_{n \in \NN}$ of i.i.d.~Steinhaus random variables on $(\Omega',\PP')$ such that 
$$
(X \oplus_Z Y)_{\infty} \cong \op{MIN}(L^{\infty}(\Omega', \PP')) \slash \overline{\op{span}}(s_n')_{n \in \NN}.
$$
\end{stw} 
\begin{proof}
Recall the definition $(X\oplus_Z Y)_{\infty}:= X \oplus_{\infty} Y \slash \{(z,-z): z \in Z\}$. If $X =L^{\infty}(\Omega_1, \PP_1)$ and $Y=L^{\infty}(\Omega_2, \PP_2)$ ($(\Omega_1, \PP_1)$ and $(\Omega_2, \PP_2)$ are just copies of $(\Omega, \PP)$) then $X \oplus_{\infty} Y \simeq L^{\infty}(\Omega_1 \sqcup \Omega_2, \PP_1 \sqcup \PP_2)$, where ``$\sqcup$'' denotes the disjoint union. We define a new probability space $(\Omega', \PP') := (\Omega_1 \sqcup \Omega_2, \frac{1}{2}(\PP_1 \sqcup \PP_2))$ and immediately obtain a complete isometry $X \oplus_{\infty} Y \simeq L^{\infty}(\Omega', \PP')$. Since $\{(z,-z): z \in Z\} = \overline{\op{span}}(s_n \sqcup -s_n)_{n \in \NN}$, we just need to check that the sequence $(s_n')_{n\in \NN}$, where $s_n':= s_n \sqcup -s_n$, has the same joint distribution as $(s_n)_{n \in \NN}$. It is, fortunately, a matter of simple computation. Let us start with finding the law of a single $s_n'$:
\begin{align}
\PP'(s_n' \in A) &= \frac{1}{2} \PP_1(s_n \in A) + \frac{1}{2} \PP_2(s_n \in -A) \notag \\
&=\frac{1}{2}(\PP_1(s_n\in A) + \PP_2(s_n \in A)) = \PP(s_n \in A), \notag
\end{align}
since the distribution of $s_n$ is symmetric. We are left to show that $s_n'$ and $s_m'$ are independent for $n \neq m$:
\begin{align}
\PP'(s_n' \in A, s_m' \in B) &= \frac{1}{2} \PP_1(s_n \in A, s_m \in B) + \frac{1}{2}\PP_2(s_n \in -A, s_m \in -B) \notag \\
&= \frac{1}{2}(\PP(s_n\in A)\PP(s_m \in B) + \PP(s_n \in -A)\PP(s_m \in -B)) \notag \\
&= \PP(s_n \in A)\PP(s_m \in B) = \PP'(s_n' \in A)\PP'(s_m' \in B), \notag
\end{align}
using the independence of $s_n$ and $s_m$.
\end{proof}
We are now ready to finish the proof of Theorem \ref{nonminamalg}:
\begin{proof}[Proof of Theorem \ref{nonminamalg}]
By the preceding proposition, $(X \oplus_Z Y)_{\infty}$ is completely isometric to the space $\op{MIN}(L^{\infty})\slash \overline{\op{span}}(s_n)_{n\in \NN}$. By Corollary \ref{minquot}, this space is not completely isomorphic to a minimal operator space. This, of course, finishes the proof of Theorem \ref{main}.
\end{proof}

We will provide new examples of non-minimal amalgamated $\ell_{\infty}$-direct sum of $L^{\infty}$-spaces (and non-maximal fibre $\ell_1$-products of $L^{1}$-spaces), using Proposition \ref{amalgfib} and Corollary \ref{lpquot}.
\begin{tw}
If $X \subset L^{\infty}(\Omega, \mu), X \subset L^{\infty}(\Omega', \mu')$ and $X$ is isomorphic to an $L^p$-space for $p < \infty$ then the amalgamated $\ell_{\infty}$-direct sum $\left(\op{MIN}(L^{\infty}(\Omega, \mu)) \oplus_{X} \op{MIN}(L^{\infty}(\Omega', \mu'))\right)_{\infty}$ is not completely isomorphic to a minimal operator space.

Suppose that a Banach space $X$ has infinite cotype or is $K$-convex. Then every fibre $\ell_1$-product $\left(\op{MAX}(L^{1}(\Omega,\mu)) \times_{X} \op{MAX}(L^{1}(\Omega', \mu'))\right)_{1}$ is not completely isomorphic to a maximal operator space.
\end{tw}
\begin{proof}
Easy combination of Proposition \ref{amalgfib}, Corollary \ref{lpquot}, Theorem \ref{KP}, and Theorem \ref{nonmax}.
\end{proof}
\begin{uw}
The preceding theorem implies Theorem \ref{nonminamalg}, nevertheless we decided to include the latter, because it sheds some light on the structure of this particular amalgamated $\ell_{\infty}$-direct sum.
\end{uw}
\begin{wn}
Suppose that $G$ is a compact abelian group and $S \subset \widehat{G}$ is a Sidon set\footnote{It means that the Fourier transform $\mathcal{F}_{S}: L^{1}(G) \to c_{0}(S)$ given by $\mathcal{F}_S(f):= \widehat{f}_{|S}$ is surjective.}. Then the fibre $\ell_1$-product $\left(\op{MAX}(L^{1}(G)) \times_{c_{0}(S)} \op{MAX}(L^{1}(G)) \right)_{1}$ is not completely isomorphic to a maximal operator space. A concrete example is given by $G = \TT$ and $S = \{2^n\}_{n \in \NN}$.
\end{wn}

\section{Acknowledgements}
I am greatly indebted to Piotr So\l tan, advisor of my Master's thesis, from which this work has emerged. I would also like to express my gratitude to Micha\l \   Wojciechowski, who made me aware of the paper \cite{KP}.


\begin{thebibliography}{9999}


\bibitem[Bl]{Bl}
D.~Blecher, \emph{A completely bounded characterization of operator algebras}, Math. Ann. \textbf{303} (1995), 227--240.

\bibitem[BP]{BP}
D.~Blecher, V.~Paulsen, \emph{Tensor products of operator spaces}, J. Funct. Anal. \textbf{99} (1991), 262--292.

\bibitem[BRS]{BRS}
D.~Blecher, Z.-J.~Ruan, A.~Sinclair, \emph{A characterization of operator algebras}, J. Funct. Anal. \textbf{89} (1990), 188--201.

\bibitem[DJT]{DJT}
J.~Diestel, H.~Jarchow, A.~Tonge, \emph{Absolutely Summing Operators}, Cambridge University Press, 1995.

\bibitem[Du]{Du}
P.~Duren, \emph{Theory of $H^{p}$ spaces}, Academic Press, 1970.

\bibitem[ER]{ER}
E.~G.~Effros, Z.-J.~Ruan, \emph{Operator Spaces}, Oxford University Press, 2000.

\bibitem[Fr]{Fr}

D.~H.~Fremlin, \emph{Measure Theory. Volume II: Broad Foundations}, Torres Fremlin, 2001.

\bibitem[Gr]{Gr}
A.~Grothendieck, \emph{R\'{e}sum\'{e} de la th\'{e}orie m\'{e}trique des produits tensoriels topologiques}, Boll. Soc. Mat. S\~{a}o-Paulo \textbf{8} (1953), 1--79.

\bibitem[KP]{KP}
N.~J. Kalton, A. Pe\l{}czy\'{n}ski, \emph{Kernels of surjections from $\mathcal{L}_1$-spaces with an application to Sidon sets}, Math. Ann. \textbf{301} (1997), 135--158.

\bibitem[LPP]{LPP}
F.~Lust-Piquard, G.~Pisier, \emph{Non-commutative Khintchine and Paley inequalities}, Ark. f\"{o}r Mat \textbf{29} (1991), 241--260.

\bibitem[Oi]{Oi}
T.~Oikhberg, \emph{Subspaces of maximal operator spaces}, Integral Equations Operator Theory \textbf{48} (2004), 81--102.

\bibitem[Pa]{Pa}
V.~Paulsen, \emph{Completely Bounded Maps and Operator Algebras}, Cambridge University Press, 2002.


\bibitem[Pi1]{Pi1}
G.~Pisier, \emph{Factorization of Linear Operators and Geometry of Banach Spaces}, American Mathematical Society, 1986.

\bibitem[Pi2]{Pi2}
G.~Pisier, \emph{Introduction to Operator Space Theory}, Cambridge University Press, 2003.


\bibitem[Ru]{Ru}
Z.-J.~Ruan, \emph{The operator amenability of $\op{A}(G)$}, Amer. J. Math. \textbf{117} (1995), 1449--1474.  


\end{thebibliography}
\end{document}